\documentclass{article}[IEEEtrand]
\usepackage{amsmath,amssymb,amsthm,graphics,graphicx}
\usepackage[left=3cm]{geometry}
\usepackage{hyperref}
\usepackage{setspace}
\usepackage{multicol}
\usepackage{comment}
\usepackage{comment}
\date{}
\newtheorem{theorem}{Theorem}[section]
\newtheorem{definition}{Definition}[section]

\newtheorem{corollary}{Corollary}[section]
\newtheorem{example}{Example}[section]
\newtheorem{proposition}{Proposition}[section]
\begin{document}
	\title{Unity Product Graph of Some Commutative Rings}
	\author{Mohammad Hassan Mudaber$^{1}$, Nor Haniza Sarmin$^{2*}$ \& Ibrahim Gambo$^{3}$\\
		$^{1,2,3}$ Department of Mathematical Sciences, Faculty of Science,\\ Universiti Teknologi Malaysia, 81310 UTM Johor Bahru, Malaysia\\ $^{*}$nhs@utm.my\\
		}
	
	\maketitle

	\begin{abstract}
A graph is an instrument which is extensively utilized to model various problems in different fields. Up to date, many graphs have been developed to represent algebraic structures, particularly rings in order to study their properties. In this article, by focusing on commutative ring $ R $, we introduce a new notion of unity product graph associated with $ R $ and its complement. In addition, we prove that if the number of vertices of the unity product graph is at least 2, then the graph is disconnected, while its complement graph is connected.  Furthermore, it is shown that there are some commutative rings with  such as Boolean ring and the Cartesian product of Boolean rings in which their associated unity product graphs are trivial. Consequently, some results are established to determine the number of isolated vertices in unity product graph. We also characterize commutative rings with unity in which their associated unity product and complement unity product graphs are empty graph and complete graph, respectively. Finally, we prove some results on the properties of the unity product graph and its complement in terms of girth, diameter, radius, dominating number, chromatic number and clique number as well as planarity and Hamiltonian.   
		
	\end{abstract}

\section{Introduction}
In the field of graph theory, an ordered pair of vertex set $ V $ and edge set $ E $ is called a graph and is denoted by $ \Gamma = (V,E) $. A finite or infinite chain of edges that connect distinct vertices is called a path. The graph $ \Gamma $ is said to be connected if there exists a path between every two vertices $ x $ and $ y$ in $V$, otherwise $ \Gamma $ is disconnected. The graph $ \Gamma $ that contains exactly one vertex is called a trivial graph. A vertex $ x $ is called isolated if there is no edges incident to it. The graph $ \Gamma $ is called empty if it contains no edges. It is called  complete if and only if every two distinct vertices are adjacent. The maximum of the shortest path between the vertex $ x $ and all other vertices in $ \Gamma $ is called the eccentricity of the vertex $ x $ and is denoted by $ ecc(x) $. The maximum of the eccentricities of all vertices in a graph $ \Gamma $ is called the diameter, denoted by $ diam (\Gamma)$. The minimum of the eccentricities of all vertices in a graph $ \Gamma $ is called the radius, denoted by $ rad(\Gamma)$. The length of the shortest cycle contained in a graph $ \Gamma $ is called the girth, denoted by $ gr(\Gamma) $. The  smallest integer $ k $ such that the graph $ \Gamma $ has a $ k $-colouring, is called the chromatic number, denoted by $ \chi( \Gamma) $. The greatest integer $ n $ such that $ K_{n}\subseteq \Gamma $, is called the clique number, denoted by $ \omega(\Gamma) $ \cite{1}. A set $ D $ of vertices of $\Gamma $ is a dominating set of $\Gamma $ if every vertex in $V(\Gamma)-D $ is adjacent to some vertex in D. The cardinality of a minimum dominating set in $\Gamma $ is called the domination number of $\Gamma $, denoted by $ \gamma(\Gamma )$. A graph $ \Gamma $ is called planar if $ \Gamma $ can be drawn in the plane such that no two of its edges cross each other. in other word, a graph $ \Gamma $ is planar if and only if $ \Gamma $ does not contain a subdivision of $ K_{5} $ or $ K_{3,3} $ as a subgraph (Kuratowski’s Theorem). The graph $ \Gamma $ is called Hamiltonian if it contains a Hamiltonian cycle (a cycle that contains every vertex of $ \Gamma $) \cite{2}.\\ 

\indent An algebraic structure $ R $ is called a ring if $ R $ is an abelian group under addition, closed and associative under multiplication and preserve the left and right distributive laws. The ring $ R $ is called commutative, if for all $ a,b\in R $, $ a\cdot b=b\cdot a $. It is called commutative ring with unity if $ R $ contains unity element 1. It is called finite if it has finite elements, otherwise it is called infinite ring. An element $ x $ of $ R $ is called a unit if there exists an element $ y\in R $ such that $ x\cdot y=e $. The least positive integer $ n $ is called the characteristics of $ R $ if $ n\cdot x =0$, for all $ x\in R $. A ring with unity in which every element is idempotent is called a Boolean ring \cite{3}. A commutative ring with unity is called a field if every non-zero element has a multiplicative inverse \cite{4}. \\

\indent The idea of studying the connection between two areas of mathematics namely, graph theory and commutative ring theory was first introduced by Beck \cite{5}. He assigned a simple graph $ \Gamma(R) $ to the commutative ring $ R $, by considering all elements of $ R $ as the vertex set and the set of all $ \{x,y\} $ that satisfy $ x\cdot y=0 $ as the edge set. Later, Anderson and Livingston \cite{6} associated a simple graph $ \Gamma(R) $ with the commutative ring $ R $ with some modification from Beck's definition of zero divisor graph to investigate the interplay of ring theoretic property of $ R $ with graph theoretic property of $ \Gamma(R) $. Recently, Sinha and Kaur \cite{7} investigated some characteristics of $ \Gamma(R) $ introduced by Beck in terms of diameter and girth and  compared the results with the results obtained by Anderson and Livingston.  Mohammadian \cite{8} continued with the research on zero divisor graph of Boolean ring and proved that the Boolean ring containing more than four elements can be determined using its unit graph $ \Gamma(R) $. Besides, he proved that $ diam (\Gamma(R))=3 $ if $ |R|>4 $. Then, the domination number of zero divisor graph $ \Gamma(R) $ was investigated in \cite{9} and for a division ring and local ring it was found that $ \gamma(R)=1 $ and $ \gamma(R)=2 $ respectively. Moreover, the relation between zero divisor  and compressed zero divisor graph was studied and it was shown that the domination number of these two graphs are equal. However, the concept of zero divisor graph associated with commutative ring has been generalized based on the notion of matrix in \cite{10}. Besides, the chromatic number, domination number and independence number of the generalized zero divisor graph over a finite field were investigated.\\

\indent Twenty years later from Beck’s definition of zero divisor graph of commutative ring, Anderson and Badawi \cite{11} defined the concept of total graph associated with commutative ring $ R $  with unity, denoted by $ T(\Gamma(R)) $. The vertices of Anderson and Badawi’s graph is the whole elements of $ R $  and two vertices  $ a $ and $ b $  of $ V(T(\Gamma))$,  $ a\neq b $ are adjacent if and only if  $ a+b\in Z(R) $, where  $  Z(R) $ is the set of zero divisor of  $ R $. Then, further study on the properties of total graph associated with commutative ring with unity were presented in \cite{12}. Chelvam and Asir \cite{13} conducted a survey on zero divisor graph and total graph associated with rings to present the distance in zero divisor graph and total graph by focusing on the diameter and girth of both graphs. Kimball and LaGrange \cite{14} generalized the concept of zero divisor graph associated with ring $ R $ by introducing the idempotent-divisor graph, denoted by $ \Gamma_{e}(R) $, where $ e $  is the idempotent element of  $ R $. Then, some properties of this graph in terms of finiteness and connectedness were investigated. The concept of zero-divisor Cayley graph associated with finite commutative ring $ R $ was introduced by Naghipour \cite{15}, denoted by  $ ZCAY(R) $. He established some results on the properties of  zero divisor Cayley graph in terms of girth, planarity, connectivity and clique number. Associate ring graph is another representation of a ring $ R $ with unity, this graph was introduced by Subhokor \cite{16}. The vertex set of this graph is $ V(\Gamma(R))=R\setminus\{0\} $ and two distinct vertices $ a $  and $ b $ in  $ V(\Gamma(R)) $ are connected if and only if $ a $  and $ b $ are associate, that is $ Or(a)=Or(b) $, where $ Or(a)$ and $Or(b) $ are the orbit of elements $ a $ and $ b $, respectively. Besides, it was shown that if  $ R $  is a boolean ring, then $ \Gamma(R) $ is empty. Moreover, For  $ R=\mathbf{Z}_{n} $, the associate ring graph is complete if $ n $  is a prime number.\\

\indent The notion of prime graph associated with the ring $ R $ was introduced in \cite{17}. It was shown that a semiprime ring is a prime ring if and only if its prime graph is a tree. Later, Pawary and Joshi \cite{18} defined the complement of prime graph. They mostly focused on finding the degree of each vertex as well as the number of triangles contain in prime and complement prime graphs of ring $ \mathbf{Z}_{n} $. \\

\indent The notion of associating the ring $ \mathbb{Z}_{n} $ with a simple graph, called unit graph of ring $ \mathbb{Z}_{n} $, was first established in\cite{19}. Later, this approach was generalized to ring $ R $ in \cite{20}; a graph whose vertices are all elements of $ R $ and two different elements $ x $ and $ y $ form an edge if and only if $ x+y $ is a unit element of $ R $ and is denoted by $ \Gamma(R) $. Then, some properties of unit graph associated with commutative rings such as girth, diameter and planarity were investigated in \cite{21,22,23}. Das et al. \cite{24} focused on the non-planarity of unit graph associated with finite commutative rings with unity and established some necessary conditions for the non-planarity of unit graph. Some mathematical proofs in terms of lemmas, propositions and theorems were proven in \cite{25} that shows $ \Gamma(R) $ is a Hamiltonian. Kiani et al. \cite{26} investigated the domination number of unit graph $ \Gamma(R) $ and classified all commutative rings with unity that have $ \gamma(\Gamma(R))<4 $. Later on, Akbari et al. \cite{27} studied the unit graph of non commutative ring and proved that if $ R $ is an artinian ring, $  2 $ is a unit element of $ R $ and the clique number of the unit graph is finite, then $ R $  is a finite ring. A non-zero element $ x $ of semi ring $ S $ is said to be semi unit if there exists, $ y,z \in S  $ such that $ 1+xy=xz $.  Ahmed and Aslam \cite{28} introduced the notion of semi unit graph associated with semiring. It is a graph whose vertices are all elements of semiring and two distinct vertices $ x $ and $ y $ are connected if and only if $ x+y $ is a semi unit. In addition, they concentrated on connectivity, diameter, girth and completeness of semi unit graph.  \\

\indent The purpose of this paper is to investigate the connection between the commutative ring theory and graph theory by introducing a new notion of graph called the unity product graph associated with a ring $ R $. It is a graph with vertex set $ U(R) $ and two distinct vertices $  x $ and $ y $ are adjacent if and only if $ x\cdot y=e $. The unit graph of a ring $ R $ introduced in \cite{20}; is a graph with vertex set $ R $ and two distinct vertices $ x $ and $ y $ are adjacent if and only if $ x+y\in U(R) $. Although these two graphs present the unit structure of the ring $ R $, but the introduced unity product graph is based on the  definition of the unit element of a ring $ R $ which is: an element $ x\in R $ is called a unit if there exists an element $ y\in R $ such that $ x\cdot y=e $. Hence, this clarifies that the unity product graph provide a better structure of the units of a ring $ R $. In addition, the results obtained by the unity product graph illustrate the real properties of the unit structure of $ R $.  
In Section 2, we introduce the definition of unity product graph, denoted by $ \Gamma^{'}(R) $. Moreover, its complement graph, denoted by $ \Gamma^{'c}(R) $ is also introduced. In Section 3, the connectivity of $ \Gamma^{'}(R) $ and its complement graph are established. Some results are  provided to determine the types of $ \Gamma^{'}(R) $ and $ \Gamma^{'c}(R) $ and also the number of isolated vertices  associated to these two graphs. In Section 4, the girth, diameter and radius properties of unity product graph and its complement are also investigated. In Section 5, the dominating number, chromatic number and clique number of the unity product graph and its complement graph are determined. Finally, in Section 6, we provide some results which show the unity product graph is planar but not Hamiltonian. In addition, some necessary and sufficient conditions are provided which show the complement unity product graph is both planar and Hamiltonian.    
 
 \section{Unity Product Graph of a Ring and Some Examples }  
 
 In this section, we define the unity product graph and the complement unity product graph associated with a ring $ R $ and provide some examples. The definition of unity product graph is fist given.

 \begin{definition}
 	Unity Product Graph \\
	Let $ U(R) $ denote the set of unit elements of a ring $ R $. Then, the unity product graph associated with $ R $, denoted by $ \Gamma^{'}(R) $, is a graph that has the vertex set  $ V(\Gamma^{'}(R))=U(R) $ and for any distinct $ r_{i},r_{j}\in V(\Gamma^{'}(R)) $, $ \{r_{i},r_{j}\} $ is an edge of $ \Gamma^{'}(R) $ if and only if $  r_{i}\cdot r_{j}=e $.
 \end{definition}

 In the next definition, we introduce the complement unity product graph of a ring $ R $.  
 \begin{definition}
 	Complement Unity Product Graph \\
 	Let $ U(R) $ denote the set of unit elements of a ring $ R $. Then, the complement unity product graph associated with $ R $, denoted by $ \Gamma^{'c}(R) $, is a graph that has the vertex set  $ V(\Gamma^{'c}(R))=U(R) $ and for any distinct $ r_{i},r_{j}\in V(\Gamma^{'c}(R)) $, $ \{r_{i},r_{j}\} $ is an edge of $ \Gamma^{'c}(R) $ if and only if $  r_{i}\cdot r_{j}\neq e $.
 \end{definition}
Now, by the concepts of the unity product graph and complement unity product graph, we observe some examples.
\begin{example}
	Let $ R=\mathbb{Z}_{11} $, then $ U(R)=\{1,2,3,4,5,6,7,8,9,10\} $. By the definition of $ \Gamma^{'}(R) $, $ V(\Gamma^{'}(R))= \{1,2,3,4,5,6,7,8,9,10\}$ and $ E(\Gamma^{'}(R))= \{\{2, 6\},\{3, 4\},\{5, 9\}, \{7, 8\}\}$. Hence, $ \Gamma^{'}(R) $ is an undirected graph with 10 vertices and 4 edges as follows: 
\end{example}
\begin{figure}[h!]
	\centerline{\includegraphics[width=5cm]{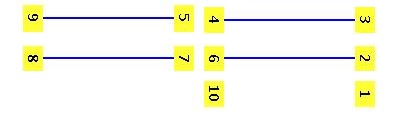}} 
\end{figure}
\begin{center}
	Figure 2.1: The unity product graph of ring  $ R= \mathbb{Z}_{11}$.
\end{center}

\begin{example}
	Let $ R=\mathbb{Z}_{11} $, then $ U(R)=\{1,2,3,4,5,6,7,8,9,10\} $. By the definition of $ \Gamma^{'c}(R) $, $ V(\Gamma^{'c}(R))= U(R)$ and $ E(\Gamma^{'}(R))= \{{r_{i},r_{j}}: r_{i}\cdot r_{j}=1\ \textrm{for all}\ r_{i}\neq r_{j}\}$. Hence, $ \Gamma^{'c}(R) $ is an undirected graph with 10 vertices and 41 edges as presented in Figure 2.2. 
\end{example}
\newpage
\begin{figure}[h!]
	\centerline{\includegraphics[width=5cm]{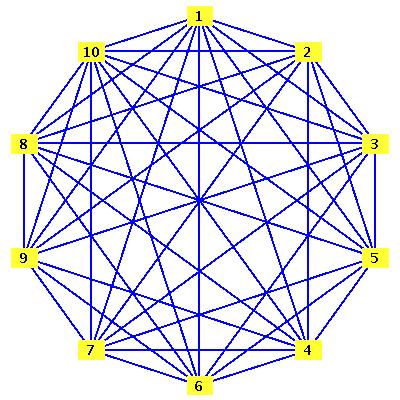}} 
\end{figure}
\begin{center}
	Figure 2.2: The complement unity product graph of ring  $ R= \mathbb{Z}_{11}$.
\end{center}

\begin{example}
	Let $ R=\mathbb{Z}_{16} $, then $ U(R)=\{1,3,5,7,9,11,13,15\} $. By the definition of $ \Gamma^{'}(R) $, $ V(\Gamma^{'}(R))= \{1,3,5,7,9,11,13,15\}$ and $ E(\Gamma^{'}(R))= \{\{3, 11\},\{5, 13\}\}$. Hence, $ \Gamma^{'}(R) $ is an undirected graph with 8 vertices and 2 edges as follows: 
\end{example}
\begin{figure}[h!]
	\centerline{\includegraphics[width=5cm]{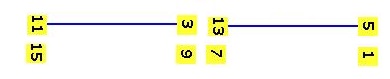}} 
\end{figure}
\begin{center}
	Figure 2.3: The unity product graph of ring  $ R= \mathbb{Z}_{16}$.
\end{center}

\begin{example}
	Let $ R=\mathbb{Z}_{16} $, then $ U(R)=\{1,3,5,7,9,11,13,15\} $. By the definition of $ \Gamma^{'c}(R) $, $ V(\Gamma^{'c}(R))= U(R)$ and $ E(\Gamma^{'}(R))= \{{r_{i},r_{j}}: r_{i}\cdot r_{j}=1\ \textrm{for all}\ r_{i}\neq r_{j}\}$. Hence, $ \Gamma^{'}(R) $ is an undirected graph with 8 vertices and 26 edges as follows: 
\end{example}
\begin{figure}[h!]
	\centerline{\includegraphics[width=5cm]{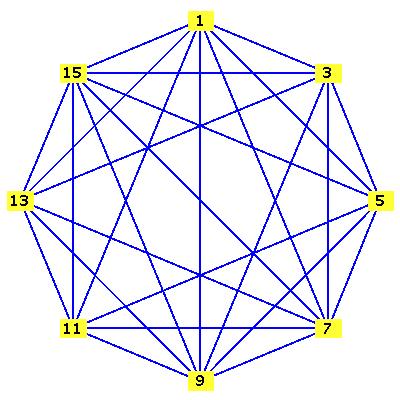}} 
\end{figure}
\begin{center}
	Figure 2.4: The complement unity product graph of ring  $ R= \mathbb{Z}_{16}$.
\end{center}
\section{Connectivity of Unity Product Graph and Its Complement}

In this section, the connectedness property of $ \Gamma^{'}(R) $ and $ \Gamma^{'c}(R) $ are investigated by establishing some theorems, propositions and corollaries. 
The first theorem states that if $ R $ is a commutative ring with unity and $ |U(R)|=1 $, then its unity product graph is a trivial graph.  
\begin{theorem}
	Let $ \Gamma^{'}(R) $ be the unity product graph associated with ring $ R $. If $ R\cong B $ or $ R\cong B\times B\times B\times \cdots \times B $, where $ B $ is a Boolean ring, then $ \Gamma^{'}(R) $ is trivial. 
\end{theorem}

\begin{proof}
	For $ R\cong B $ the proof is obvious. Assume $ R\cong B\times B\times B\times \cdots \times B$ is given as the Cartesian product of $ n $ Boolean rings, which contains unity element $ e= (1,1,1,\cdots ,1)$. Thus, $ U(R)=\{r_{i}\in R: r_{i}\cdot r_{j}=(1,1,1,\cdots ,1), \ \textrm{where} \ 1\leq i\leq j\leq n\} $. By the definition of $ \Gamma^{'}(R) $, the vertex set $ V(\Gamma^{'}(R))= U(R)=\{e\}$ and the edge set $ E(\Gamma^{'}(R))=\emptyset$. Hence, $ \Gamma^{'}(R) $ is a graph containing exactly one vertex, which is a trivial graph.  
\end{proof}
The result obtained in Theorem 3.1 can be generalized to a unity product graph associated with ring $ R $ which is isomorphic to the Cartesian product of infinite Boolean ring $ R\cong B\times B\times B\times \cdots $ as presented in the following corollary. 
\begin{corollary}
	Let $ \Gamma^{'}(R) $ be the unity product graph associated with ring $ R $. If $ R\cong B\times B\times B\times \cdots  $, where $ B $ is a Boolean ring, then $ \Gamma^{'}(R) $ is trivial.
\end{corollary}

The results are presented in Corollary 3.2 as a consequence of Theorem 3.1 and Corollary 3.1 on the complement unity product graph, $ \Gamma^{'c}(R) $.

\begin{corollary}
	Let $ \Gamma^{'c}(R) $ be the complement unity product graph associated with a ring $ R $. Then $ \Gamma^{'c}(R) $ is trivial if $ R $ is isomorphic to any of the following Boolean rings:\\
	(1) $ R\cong B$\\
	(2) $ R\cong B\times B\times B\times \cdots \times B$\\
	(3) $ R\cong B\times B\times B\times \cdots $
\end{corollary}
\begin{theorem}
Let $ \Gamma^{'}(R) $ be the unity product graph associated with a finite commutative ring $ R $ with unity. If $ |V(\Gamma^{'}(R))| \geq 2$, then $ \Gamma^{'}(R) $ is disconnected.
\end{theorem}
\begin{proof}
Assume $ R $ is a finite commutative ring with unity, thus $ R $ contains finite elements as $ R=\{r_{i}:i=1,2,3,\cdots ,n\} $. Suppose that $ U(R) $ denote the set of units of $ R $, by the definition of unity product graph, $ V(\Gamma^{'}(R))=U(R) $. 
 Assume $ |V(\Gamma^{'}(R))| \geq 2 $. Since $ R $ is a commutative ring with unity, thus there exists a vertex $ e$ in $ V(\Gamma^{'}(R))$ such that $e\cdot r_{i}\neq 1$ for all $ e\neq r_{i}\in V(\Gamma^{'}(R))$. This shows that there is no path between the vertex $ e $ and the vertex $ r_{i} $ in $ \Gamma^{'}(R) $. Similarly, for every $ e\neq r_{i}\in V(\Gamma^{'}(R)) $, there is a unique $ r_{j}\in V(\Gamma^{'}(R)) $, such that $ r_{i}\cdot r_{j}=e $. This shows that $ r_{i} $ is only adjacent to $ r_{j} $. If $ r_{m}\in  V(\Gamma^{'}(R)) $, where $ i<m<j $, then $ r_{m} $ is not connected to $ r_{i} $, since  $ r_{i}\cdot r_{m}\neq e $. Hence, $ \Gamma^{'}(R) $ is disconnected.
\end{proof}

The result obtained in Theorem 3.2 can be applied to the unity product graph associated with an infinite commutative ring with unity. This result is stated in the following corollary. 
 \begin{corollary}
	Let $ \Gamma^{'}(R) $ be the unity product graph associated with an infinite commutative ring $ R $ with unity. If $ |V(\Gamma^{'}(R))| \geq 2$, then $ \Gamma^{'}(R) $ is disconnected.
\end{corollary}

\begin{theorem}
	Let $ \Gamma^{'c}(R) $ be the complement unity product graph associated with a finite commutative ring $ R $ with unity. If $ |V(\Gamma^{'c}(R))| \geq 2$, then $ \Gamma^{'c}(R) $ is connected. 
\end{theorem}

\begin{proof}
	Suppose that $ R $ is a finite commutative ring with unity, i.e. $ R=\{r_{i}:i=1,2,3,\cdots ,n\} $. Let $ U(R) $ denote the set of units of $ R $, thus by the definition of $ \Gamma^{'c}(R) $,  $ V(\Gamma^{'c}(R)=U(R) $. According to vertices adjacency property of $ \Gamma^{'c}(R) $, $ \{r_{i},r_{j}\} $ is an edge of $ \Gamma^{'c}(R) $ if and only if $ r_{i}\cdot r_{j}\neq e $. Since, $ e $ is an isolated vertex in $ \Gamma^{'}(R) $, thus $ e\cdot r_{i}\neq e $ for all $ e\neq r_{i}\in V(\Gamma^{'c}(R)$. This shows that there is a path between the vertex $ e $ and the vertex $ r_{i} $ in  $ \Gamma^{'c}(R) $. Hence, $ \Gamma^{'c}(R) $ is connected.  
\end{proof} 

The result obtained in Theorem 3.3 can be applied to the unity product graph associated with an infinite commutative ring with unity. This result is stated in the following corollary. 
\begin{corollary}
	Let $ \Gamma^{'c}(R) $ be the complement unity product graph associated with an infinite commutative ring $ R $ with unity. If $ |V(\Gamma^{'c}(R))| \geq 2$, then $ \Gamma^{'c}(R) $ is connected. 	
\end{corollary}

In the following example, we show that $ \Gamma^{'}(R) $ is connected  and  $ \Gamma^{'c}(R) $ is  disconnected, for a commutative ring without unity.
\begin{example}	
Let $ R=\{ai:a\in \mathbb{Z},\ i^{2}=-1\} $. Then, $ U(R)=\{-i,i\} $. By the definition of unity product graph, $ V(\Gamma^{'}(R))= \{-i,i\} $ and $ E(\Gamma^{'}(R))= \{\{-i,i\}\} $. Therefore, $ \Gamma^{'}(R) $ is connected. However, according to the definition of $ \Gamma^{'c}(R) $, the vertex set $ V(\Gamma^{'c}(R))= \{-i,i\} $ and the edge set $ E(\Gamma^{'c}(R))= \emptyset $. Therefore, $ \Gamma^{'c}(R) $ is disconnected. 
\end{example}
\begin{theorem}
	Let $ \Gamma^{'}(R) $ be a unity product graph associated with a ring $ R $. Then, $ \Gamma^{'}(R) $ contains two isolated vertices, if $ R $ is isomorphic to any of the following:\\
	(1) $ R\cong \mathbb{Z}_{p}$, where $ p$ is an odd prime.\\
	(2) $ R\cong \mathbb{F} $, where $\mathbb{F} $ is a field with $ Char(\mathbb{F})=0 $. 
\end{theorem}
\begin{proof}
 (1). Assume $ R\cong \mathbb{Z}_{p}$, where $ p$ is an odd prime which contains the unity element $ e $. Thus $ R $ can be displayed as $ R=\{r_{i}: i=1,2,3,\cdots ,p\}$. Since $ p $ is an odd prime, thus $ |U(R)|=p-1 $. According to the definition of $ \Gamma^{'}(R) $, $ V(\Gamma^{'}(R))=U(R)$.  
 Since $ e\in V(\Gamma^{'}(R))$, thus $ e\cdot r_{i}\neq e $ for all $ e\neq r_{i}\in V(\Gamma^{'}(R))$. Therefore $ e $ is an isolated vertex. Since $ p-1 $ is even, hence there exist $ \frac{p-3}{2}$ pairs of vertices $ r_{i} $ and $ r_{j} $ such that $ r_{i}\cdot r_{j}=e $. This implies that $ r_{p} $ is isolated. \\
(2). Assume $ R\cong \mathbb{F} $, where $ \mathbb{F} $ is a field with $ Char(\mathbb{F})=0 $. Thus $ R $ is represented as $ R=\{r_{i}\}_{i=1}^{\infty}$. By the definition of a field, every non-zero $ r_{i} $ has a unique multiplicative inverse, thus $ |U(R)|=\infty$. According to the definition of $ \Gamma^{'}(R) $, $ V(\Gamma^{'}(R))=U(R)$.  
 Since $ e\in V(\Gamma^{'}(R)) $ and $ e\cdot r_{i}\neq e $ for all $ e\neq r_{i}\in V(\Gamma^{'}(R))$, therefore $ e $ is an isolated vertex. Similarly, there exists a vertex $ -e\in V(\Gamma^{'}(R)) $  such that $ -e\cdot  r_{i}\neq e $ for all $ -e\neq r_{i}\in V(\Gamma^{'}(R))$.  Hence, $ \Gamma^{'}(R) $ contains two isolated vertices of $ e$ and  $ -e $, respectively. 
\end{proof}
 \begin{theorem}
 	Let $ \Gamma^{'}(R) $ be a unity product graph associated with the ring $ R $. If $ R\cong \mathbb{Z}_{2^m} $, where $ m \geq 3 $, then $ \Gamma^{'}(R) $ contains four isolated vertices.
 \end{theorem}
\begin{proof}
	Suppose $ R\cong \mathbb{Z}_{2^m} $, where $ m \geq 3 $.  Thus $ |U(R)|=2^{m-1} $. By the definition of unity product graph, $ V(\Gamma^{'}(R))=\{r_{i}:r_{i}\cdot r_{j}=e \ \textrm{for all} \ r_{i}\neq r_{j}\} $, which has $ 2^{m-1} $ elements. It follows from Theorem 3.4 (1) that $ e $ and $ r_{2^m} $ are isolated vertices. Since the number of vertices of $ \Gamma^{'}(R $ is equal to $ 2^{m-1} $, which is even for all $ m \geq 3 $, thus there exist two vertices $ r_{k} $ and $ r_{l} $ in $ V(\Gamma^{'}(R))  $ such that $ \frac{r_{k}+r_{l}}{2} $ is the median of all elements in $ V(\Gamma^{'}(R)) $ and satisfy $ r_{k}\cdot r_{k}=e $ and $ r_{l}\cdot r_{l}=e $. Consequently,  $ \Gamma^{'}(R) $ contains four isolated vertices of $ e $, $ r_{k} $, $ r_{l} $ and $ r_{2^m} $, respectively.
\end{proof}
\begin{proposition}
		Let $ R\cong \mathbb{Z}_{n}$ be a ring and $U(R)$ be the set of units of $ R $. If $U(R)$ contains no composite elements, then the elements of $U(R) $ are self-inverses.    
\end{proposition}
\begin{proof}
	Assume $ R\cong \mathbb{Z}_{n}$, thus $ R $ with respect to its elements can be represented as $ R=\{r_{i}: i=1,2,3,\cdots, n\} $. Further assume that $U(R)$ denote the set of units of $ R $, which does not contain any composite elements, i.e. $U(R)=\{r_{i}: r_{i}\cdot r_{j}=e, \ \textrm{where} \ r_{i} \ \textrm{is not composite} \ \textrm{for} \ 1\leq i\leq j\leq n\} $. Since $ r_{i}\in U(R)$ is not a composite, thus for all $ r_{i},r_{j}\in U(R)$ such that $ r_{i}\neq r_{j} $ yields that $ r_{i}\cdot r_{j}\neq e $. This implies that $r_{i}$ and  $ r_{j} $ are self-inverses.   
\end{proof} 

The following example illustrates that if $U(R)$ contains at least one composite element, then there exist  element(s) in $U(R)$ which is/are not self-inverse(s). 
\begin{example}
	Let $ R\cong \mathbb{Z}_{14}$, i.e. $ R $ has 14 distinct elements as $ R=\{0,1,2,3,4,5,6,7,8,9,10,11,12,13\} $. Assume $U(R)$ is the set of units of $ R $, thus $U(R)=\{1,3,5,9,11,13\}$. Since $ 9 $ is a composite, thus $ 3\cdot 5=1 $ and $ 9\cdot 11=1 $. Hence, the elements $ 3,5,9 $ and $ 11 $ are not self-inverses. 
\end{example}
\begin{proposition}
	Let $ R $ be a ring and $ \Gamma^{'}(R) $ be the unity product graph associated with $ R $. Then $  \Gamma^{'}(R) $ is an empty graph if $ R $ is isomorphic to any of the following rings:\\ 
	 (1) $ R\cong S $, where $ S $ is a reduced ring with $ Char(S)=0 $ and $ |U(S)|=2 $.\\
	 (2) $ R\cong \mathbb{Z}_{n}$, where $ n>1$ is a divisor of 24.
\end{proposition}
\begin{proof}
(1). Let $ R\cong S $, where $ S $ is a reduced ring with $ Char(S)=0 $ and $ |U(S)|=2 $. By the definition of unity product graph, $ V(\Gamma^{'}(R))=U(R)$. Since $e\in V(\Gamma^{'}(R))$ such that $ e\cdot r_{i}\neq e $ for $ e\neq r_{i}\in V(\Gamma^{'}(R)) $,  thus $ e $ is an isolated vertex. Suppose that $ e^{'} $ is another element of $ V(\Gamma^{'}(R)) $, thus $ {e}^{'} $ is also an isolated vertex because  $ {e}\cdot {e}^{'}\neq e$. This shows that $ E(\Gamma^{'}(R))=\emptyset$. Therefore, $ \Gamma^{'}(R) $ is an empty graph.\\
(2). Assume $ R\cong \mathbb{Z}_{n}$, therefore $ R $ can be represented as $ R=\{r_{i}: i=1,2,3,\cdots, n\} $. According to the hypothesis of the theorem, $ n\neq 1$ and $ n $ is a divisor of 24, thus $ U(R)=\{r_{i}: r_{i}\ \textrm{is not composite}\} $. Since $ U(R) $ does not contain composite elements, by Proposition 3.1, the elements of $ R $ are self-inverses. This implies that $ V(\Gamma^{'}(R))=U(R)$ and $ E(\Gamma^{'}(R))=\emptyset$. Thus, $ \Gamma^{'}(R) $ is an empty graph. 
\end{proof}

The following counterexample demonstrates that if $ n\nmid 24 $, then $ \Gamma^{'}(R) $ is not an empty graph. 
\begin{example}
	Let $ R\cong \mathbb{Z}_{n}$, where $ n $ is not a divisor of 24. Thus $ n\notin \{1,2,3,4,6,8,12,24\} $. Suppose $ n=5 $, therefore $ R=\{0,1,2,3,4\} $. Let $ U(R) $ denote the set of units of $ R $, then $ U(R)= \{1,2,3,4\} $. By the definition of unity product graph, $ V(\Gamma^{'}(R))=\{1,2,3,4\}$ and  $ E(\Gamma^{'}(R))=\{\{2,3\}\}$. Hence, $ \Gamma^{'}(R) $ is not an empty graph. 
\end{example}
\begin{proposition}
	Let $ R $ be a ring and $ \Gamma^{'c}(R) $ be the complement unity product graph associated with $ R $. Then $ \Gamma^{'c}(R) $ is a complete graph if $ R $ is isomorphic to any of the following rings: \\
	(1)  $ R\cong S $, where $ S $ is a reduced ring with $ Char(S)=0 $ and $ |U(S)|=2 $.\\
	(2) $ R\cong \mathbb{Z}_{n}$, where $ n>2$ is a divisor of 24.
\end{proposition}
\begin{proof}
	(1). Assume $ R\cong S $, where $ S $ is a reduced ring with $ Char(S)=0 $ and $ |U(S)|=2 $. According to Proposition 3.2 (1), $ \Gamma^{'}(R) $ is $ \bar{K}_{2} $ graph. Hence, $ \Gamma^{'c}(R) $ is a complete graph $ K_{2} $.\\
	(2). Let $ R\cong \mathbb{Z}_{n}$, where $ n>2$ is a divisor of 24. Thus, according to Proposition 3.2 (2), $ \Gamma^{'}(R) $ is a graph with $ V(\Gamma^{'}(R))=U(R) $ and $ E(\Gamma^{'}(R))=\emptyset $. Since, $ \Gamma^{'}(R) $ is an empty graph, therefore $ \Gamma^{'c}(R) $ is a complete graph.   
\end{proof}
\begin{theorem}
	Let $ R $ be a finite commutative ring with unity consists of $ m $ distinct sets of mutual inverses $ \{r_{i},r_{j}\} $, where $ 1\leq i\leq j\leq m $ of order 1 or 2. Then, the unity product graph is either $ \Gamma^{'}(R)= 2K_{1}+(m-2)K_{2}$ or  $ \Gamma^{'}(R)= 4K_{1}+(m-4)K_{2}$ or $ \Gamma^{'}(R)= mK_{1}$.
\end{theorem}
\begin{proof}
	Assume $ R $ is a finite commutative ring with unity $ e $, that is $ R=\{r_{i}: i=1,2,3,\cdots, n \}$. Since $ R $ consists of $ m $ distinct sets of mutual inverses $ \{r_{i},r_{j}\} $, where $ 1\leq i\leq j\leq m $ of order 1 or 2, this implies that $ \Gamma^{'}(R)= \bigcup_{i=1}^{m}K_{m_{i}}$. By Theorem 3.4, Theorem 3.5 and Proposition 3.2, $ \Gamma^{'}(R) $ either contains 2 or 4 isolated vertices or all the vertices are isolated. Let $\Gamma^{'}(R) $ contains 2 isolated vertices,  thus $ \Gamma^{'}(R)= \bigcup_{i=1}^{m-2}K_{m_{i}}\cup \bigcup_{j=1}^{2}K_{m_{j}}$, where $ |m_{j}|=1 $ and $ |m_{i}|=2 $. Hence, $ \Gamma^{'}(R)= 2K_{1}+(m-2)K_{2}$.  Suppose that $\Gamma^{'}(R) $ contains 4 isolated vertices, thus $ \Gamma^{'}(R)= \bigcup_{i=1}^{m-4}K_{m_{i}}\cup \bigcup_{j=1}^{4}K_{m_{j}}$, where $ |m_{j}|=1 $ and $ |m_{i}|=2 $. Hence, $ \Gamma^{'}(R)= 4K_{1}+(m-4)K_{2}$. If all the vertices of $ \Gamma^{'}(R) $ are isolated, then $\Gamma^{'}(R)= \bigcup_{j=1}^{m}K_{m_{j}}$, where $ |m_{j}|=1 $. Hence, $ \Gamma^{'}(R)= mK_{1}$.     
\end{proof}

\begin{theorem}
	Let $ R $ be a finite commutative ring with unity consists of $ m $ distinct sets of mutual inverses $ \{r_{i},r_{j}\} $, where $ 1\leq i\leq j\leq m $ of order 1 or 2. Then, the complement unity product graph is either a complete $ m $-partite graph as $ \Gamma^{'c}(R)= K_{2,2,2,\cdots, 2,1,1}$ or  $ \Gamma^{'c}(R)= K_{2,2,2,\cdots, 2,1,1,1,1}$ or a complete graph as $ \Gamma^{'c}(R)= K_{m}$.
\end{theorem}
\begin{proof}
	Assume $ R $ is a finite commutative ring with unity consists of $ m $ distinct sets of mutual inverses $ \{r_{i},r_{j}\} $, where $ 1\leq i\leq j\leq m $ of order 1 or 2. If $ r_{i}=r_{j} $, meaning that $ |\{r_{i}r_{j}\}|=1 $, then the vertex $ r_{i} $ is adjacent with all vertices in $ \Gamma^{'c}(R) $. Hence, $ k_{1}=\{r_{i},r_{j}\} $ form a partite set of order 1. However if $ r_{i}\neq r_{j} $, meaning that $ |\{r_{i}r_{j}\}|=2 $, then the vertex $ r_{i} $ is not adjacent to $ r_{j} $, but these two vertices are adjacent with all other vertices in $ \Gamma^{'c}(R) $. Hence $ k_{2}=\{r_{i}, r_{j}\} $ form a partite set of order 2. Suppose that $ n(k_{1}) =2$, then $ n(k_{2}) =m-2$. This yields a complete $ m $-partite graph  $ \Gamma^{'c}(R)= K_{2,2,2,\cdots, 2,1,1}$. If $ n(k_{1}) =4$, then $ n(k_{2}) =m-4$. This yields a complete $ m $-partite graph as $ \Gamma^{'c}(R)= K_{2,2,2,\cdots, 2,1,1,1,1}$. However, if  $ n(k_{1}) =m$, this gives a complete graph $ \Gamma^{'c}(R)= K_{1,1,1,\cdots, 1}=K_{m}$.   
\end{proof}

\section{Girth, Radius and Diameter of the Unity Product Graph and Its Complement}
In this section, some properties of unity product graph associated with commutative rings with unity and its complement are presented in terms of girth, radius and diameter. 
 \begin{theorem}
 	Let $ R $ be a finite commutative ring with unity and $ \Gamma^{'}(R) $ be the unity product graph associated with $ R $. Then $ gr(\Gamma^{'}(R)) =\infty$. 
 \end{theorem}
 \begin{proof}
 	Assume $ \Gamma^{'}(R) $  is the unity product graph associated with $ R $. By Theorem 3.6, the unity product graph is either $ \Gamma^{'}(R)= 2K_{1}+(m-2)K_{2}$ or  $ \Gamma^{'}(R)= 4K_{1}+(m-4)K_{2}$ or $ \Gamma^{'}(R)= mK_{1}$. This implies that $ \Gamma^{'}(R) $ does not contain any cycle. Hence, $ gr(\Gamma^{'}(R)) =\infty$. 
 \end{proof}
The result established in Theorem 4.1 can be generalized to an infinite commutative ring $ R $ with unity, that is ring $ R $ with $ Char(R)=0 $. This result is presented in the following corollary. 
\begin{corollary}
	Let $ R $ be a commutative ring with $ Char(R)=0 $ and $ \Gamma^{'}(R) $ be the unity product graph associated with $ R $. Then $ gr(\Gamma^{'}(R)) =\infty$. 
\end{corollary}

 \begin{theorem}
 	Let $ R $ be a commutative ring with unity and $ \Gamma^{'c}(R) $ be the complement unity product graph associated with $ R $. If $ |U(R)|\leq 2$, then $ gr(\Gamma^{'c}(R))=\infty$. 
 \end{theorem}
 \begin{proof}
 	Assume $ R $ is a commutative ring with unity. If $ |U(R)|\leq 2$, then $ \Gamma^{'}(R) $ is either a trivial graph or an empty graph. Accordingly, $ \Gamma^{'c}(R) $ is a trivial graph or a path of length 2, which does not contain any cycle. Hence, $gr(\Gamma^{'c}(R)) =\infty$.  
 \end{proof}

\begin{theorem}
	Let $ R $ be a commutative ring with unity and $ \Gamma^{'c}(R) $ be the complement unity product graph associated with $ R $. If $ |U(R)|> 3$, then $ gr(\Gamma^{'c}(R))=3$. 
\end{theorem}
\begin{proof}
	Assume $ R $ is a commutative ring with unity. If $ |U(R)|> 3$, then according to Theorem 3.6, $ \Gamma^{'}(R) $ contains at least two isolated vertices. Let $ r_{l} $ and $ r_{2} $ be two isolated vertices in graph $ \Gamma^{'}(R) $,  then $ \Gamma^{'c}(R) $ contains at least a cycle of length 3 as $  r_{1}- r_{2}- r_{i} $, where $ r_{i}\neq r_{1} $ and $ r_{i}\neq r_{2} $, since $ r_{i}\cdot r_{1} \neq e$, $ r_{i}\cdot r_{2} \neq e$ and $ r_{1}\cdot r_{2} \neq e$. Therefore, $gr(\Gamma^{'c}(R)) =3$.
\end{proof}

 \begin{theorem}
 	Let $ R $ be a finite commutative ring with unity and $ \Gamma^{'}(R) $ be the unity product graph associated with $ R $. Then, the following holds:\\
 (1) $diam(\Gamma^{'}(R))=\infty$.\\
 (2)  $ rad(\Gamma^{'}(R))=\infty$.
 \end{theorem}
 \begin{proof}
 	(1) Assume $ R $ is a finite commutative ring with unity $ e $. By Theorem 3.6, the unity product graph is either $ \Gamma^{'}(R)= 2K_{1}+(m-2)K_{2}$ or  $ \Gamma^{'}(R)= 4K_{1}+(m-4)K_{2}$ or $ \Gamma^{'}(R)= mK_{1}$. According to the vertices adjacency property of $ \Gamma^{'}(R) $, $ d(r_{i},r_{j})=1 $ since $ r_{i}\cdot r_{j}=e $ for all $ r_{i}\neq r_{j} $. Suppose $ r_{m}\in V(\Gamma^{'}(R)) $, where $ i< m < j $, then $ d(r_{m},r_{i})=d(r_{m},r_{j})=\infty$, since $ r_{i}\cdot r_{m}\neq e $ and $ r_{j}\cdot r_{m}\neq e $. Similarly, since $ e $ is an isolated vertex in $ V(\Gamma^{'}(R)) $ such that $ e\cdot r_{i}\neq e $, therefore $ d(e,r_{i})=\infty $ for all $ e\neq r_{i}\in V(\Gamma^{'}(R))$. This implies that $ ecc(r_{i})=\infty $. Consequently, $ diam(\Gamma^{'}(R))=\infty$. \\
 	(2) Since the minimum of eccentricity of all vertices $ r_{i} $ is $ \infty $, hence $ rad(\Gamma^{'}(R))=\infty$.
 \end{proof}
The result established in Theorem 4.4 can be generalized to a commutative ring $ R $ with $ Char(R)=0$. This result is presented in the following corollary. 

   \begin{corollary}
  	Let $ R $ be a commutative ring with  $ Char(R)=0 $ and $ \Gamma^{'}(R) $ be the unity product graph associated with $ R $. Then the following holds:\\
  	(1) $diam(\Gamma^{'}(R))=\infty$.\\
  	(2)  $ rad(\Gamma^{'}(R))=\infty$.
  \end{corollary} 

 \begin{proposition}
 	Let $ R $ be a ring and $ \Gamma^{'c}(R) $ be the complement unity product graph associated with $ R $. Then $  diam(\Gamma^{'c}(R))=rad(\Gamma^{'c}(R))=1 $ if and only if  $ R $ is isomorphic to any of the following rings:\\ 
 	(1) $ R\cong S $, where $ S $ is a reduced ring with $ Char(S)=0 $ and $ |U(S)|=2 $.\\
 	(3) $ R\cong \mathbb{Z}_{n}$, where $ n>2$ is a divisor of 24.
 \end{proposition}
 
 \begin{proof}
 	(1) Assume $ R\cong S $, where $ S $ is a reduced ring with $ Char(S)=0 $ and $ |U(S)|=2 $. Since $ e\in U(R) $ is an isolated vertex in $ \Gamma^{'}(R) $, this implies that $ \Gamma^{'c}(R) $ is $ K_{2} $ graph. Consequently, the maximum and minimum eccentricity of every vertices in $ \Gamma^{'c}(R) $ is equal to 1. Therefore,  $  diam(\Gamma^{'}(R))=rad(\Gamma^{'}(R))=1 $. \\
 	(2) Assume $ R\cong \mathbb{Z}_{n}$, where $ n>2$ is a divisor of 24. According to Proposition 3.3 (2), $ \Gamma^{'c}(R) $ is a complete graph. Hence, the maximum and minimum eccentricity of every vertices in $ \Gamma^{'c}(R) $ is equal to 1. Therefore,  $  diam(\Gamma^{'c}(R))= rad(\Gamma^{'c}(R))=1 $.\\
 	Conversely, if  $  diam(\Gamma^{'c}(R))= rad(\Gamma^{'c}(R))=1 $, that is the maximum and minimum of eccentricity of every vertices are equal to 1, i.e. for all $ r_{i},r_{j}\in V(\Gamma^{'c}(R)) $, then $ d(r_{i},r_{j})=1 $. Hence, $ \Gamma^{'c}(R) $ is a complete graph. This implies that $ \Gamma^{'}(R)$ is an empty graph. Accordingly, $ |U(R)|=2 $ or $ U(R) $ does not contain any composite element. Therefore, if $ |U(R)|=2 $, then $ R\cong S $, where $ S $ is a reduced ring with $ Char(S)=0 $ and $ |U(S)|=2 $. However, if $ U(R) $ does not contain any composite elements, then $ R\cong \mathbb{Z}_{n}$, where $ n>2$ is a divisor of 24. 
 \end{proof}

 \begin{theorem}
 	Let $ R $ be a finite commutative ring with unity and $ \Gamma^{'c}(R) $ be the complement unity product graph associated with $ R $, which is not a complete graph. Then, the following holds:\\
 (1) $diam(\Gamma^{'c}(R))=2$.\\
 (2) $rad(\Gamma^{'c}(R))=1$. 
 \end{theorem} 

 \begin{proof}
 	(1) Suppose that $ R $ is a finite commutative ring with unity, thus $ R=\{r_{i}: i=1,2,3,\cdots, n\} $. Since  $ \Gamma^{'c}(R) $ is the complement unity product graph associated with $ R $, which is not a complete graph, then by Theorem 3.6, $ \Gamma^{'}(R)= 2K_{1}+(m-2)K_{2}$ or  $ \Gamma^{'}(R)= 4K_{1}+(m-4)K_{2}$. Suppose $ \{r_{i},r_{j} \}$ forms a $ K_{2} $ and $ r_{k} \in V(\Gamma^{'}(R))$  is an isolated vertex, thus $ r_{i} $ is not adjacent to $ r_{j} $  in $ \Gamma^{'c}(R) $ and the isolated vertex $ r_{k} $ is adjacent to every vertices in $ \Gamma^{'c}(R) $. This implies that, $ d(r_{i},r_{j})=2$ and $ d(r_{k},r_{i})=1$. Thus the maximum of eccentricity of all vertices in  $ \Gamma^{'c}(R) $ is 2. Hence, $ diam(\Gamma^{'c}(R))=2$.\\
 	(2) Since $ d(r_{k},r_{i})=1$, the minimum of eccentricity of all vertices in  $ \Gamma^{'c}(R) $ is 1. Therefore, $ rad(\Gamma^{'c}(R))=1$.  
\end{proof}
Theorem 4.5 yields the following result as stated in Corollary 4.3. 
\begin{corollary}
	Let $ R $ be a commutative ring with $ Char(R)=0 $ and $ \Gamma^{'c}(R) $ be the complement unity product graph associated with $ R $. If $ |U(R)|=\infty $, then the following holds:\\
	(1) $diam(\Gamma^{'c}(R))=2$.\\
	(2) $rad(\Gamma^{'c}(R))=1$. 
\end{corollary}

\section{Dominating, Chromatic and Clique Numbers of the Unity Product Graph and Its Complement}
In this section, some properties of the unity product graph associated with commutative rings with unity and its complement are presented in terms of dominating number, chromatic number and clique number.
\begin{theorem}
	Let $ R $ be a finite commutative ring with unity consists of $ m $ distinct sets of mutual inverses $ \{r_{i},r_{j}\} $, where $ 1\leq i\leq j\leq m $ of order 1 or 2. Then $ \gamma(\Gamma^{'}(R))=m$.
\end{theorem}
\begin{proof}
		Assume  $ R $ is a finite commutative ring with unity consists of $ m $ distinct sets of mutual inverses $ \{r_{i},r_{j}\} $, where $ 1\leq i\leq j\leq m $ of order 1 or 2. By Theorem 3.6, the unity product graph is either $ \Gamma^{'}(R)= 2K_{1}+(m-2)K_{2}$ or  $ \Gamma^{'}(R)= 4K_{1}+(m-4)K_{2}$ or $ \Gamma^{'}(R)= mK_{1}$. This shows that the minimum dominating set contains $ m $ elements. Hence, $ \gamma(\Gamma^{'}(R))=m$.
\end{proof}
\begin{theorem}
	Let $ \Gamma^{'}(R) $ be a unity product graph associated with a division ring $ R $ with $ Char(R)=0 $. Then, $ \gamma(\Gamma^{'}(R))=\infty$. 
\end{theorem}
\begin{proof}
	Assume $ R $ is a division ring with $ Char(R)=0 $. By Theorem 3.4(2), $ \Gamma^{'}(R) $ is a graph containing infinite copies of $ K_{2} $ with two isolated vertices, i.e. $ \Gamma^{'}(R)= \bigcup_{i=1}^{\infty} K_{2}\cup \bar{K}_{2}$. This implies that the dominating set of $ \Gamma^{'}(R) $ contains infinite elements. Hence, $ \gamma(\Gamma^{'}(R))=\infty$.   
\end{proof}

 Theorem 5.2 leads the following general results on dominating number of the unity product graph.
\begin{corollary}
	Let $ R $ be a commutative ring with  $ Char(R)=0 $ consists of infinite distinct sets of mutual inverses $ \{r_{i},r_{j}\} $, where $ 1\leq i\leq j< \infty $ of order 1 or 2. Then $ \gamma(\Gamma^{'}(R))=\infty$.
\end{corollary}

\begin{theorem}
	Let $ R $ be a finite commutative ring with unity and $ \Gamma^{'c}(R) $ be the complement unity product graph associated with $ R $. Then, $ \gamma(\Gamma^{'c}(R))=1$.  
\end{theorem}
\begin{proof}
	Assume $ R $ is a finite commutative ring with unity $e$ and $ \Gamma^{'c}(R) $ is the complement unity product graph associated with $ R $. Then $ V(\Gamma^{'c}(R))=U(R) $. Let  $ |V(\Gamma^{'c}(R))|=m $, since $ e\cdot r_{i}\neq 1 $ for any $ e\neq r_{i}\in V(\Gamma^{'c}(R))$, thus the vertex $ e $ has degree $ m-1 $. This implies that the dominating set of $ \Gamma^{'c}(R) $ contains only one element. Therefore,  $ \gamma(\Gamma^{'c}(R))=1$.  
\end{proof}

The result obtained in Theorem 5.3 can be applied to the complement unity product graph associated with an infinite commutative ring with unity. This result is stated in the following corollary.
	\begin{corollary}
		Let $ R $ be a commutative ring with $ Char(R)=0 $ and $ \Gamma^{'c}(R) $ be the complement unity product graph associated with $ R $. Then, $ \gamma(\Gamma^{'c}(R))=1$.  
	\end{corollary}	
\begin{theorem}
	Let $ R $ be a finite commutative ring with unity consists of $ m $ distinct sets of mutual inverses $ \{r_{i},r_{j}\} $, where $ 1\leq i\leq j\leq m $ of order 1 or 2 respectively. Then $ \omega(\Gamma^{'}(R))\in \{2,m\}$.
\end{theorem}
\begin{proof}
Assume $ R $ is a finite commutative ring with with unity $ e $, that is $ R $ has $ n $ distinct elements as $ R=\{r_{i}:i=1,2,3,\cdots, n\}$. Let $ U(R) $ denote the set of units of $ R $, thus  $ U(R)=\{r_{i}: r_{i}\cdot r_{j} =e\} $. Since $ U(R) $ consists of $ m $ distinct sets of mutual inverses $ \{r_{i},r_{j}\} $, where $ 1\leq i\leq j\leq m $ of order 1 or 2, therefore by Theorem 3.6, the unity product graph is either $ \Gamma^{'}(R)= 2K_{1}+(m-2)K_{2}$ or  $ \Gamma^{'}(R)= 4K_{1}+(m-4)K_{2}$ or $ \Gamma^{'}(R)= mK_{1}$. If $ \Gamma^{'}(R)= 2K_{1}+(m-2)K_{2}$ or  $ \Gamma^{'}(R)= 4K_{1}+(m-4)K_{2}$, then the largest complete subgraph of $ \Gamma^{'}(R)$ is $ K_{2} $. This shows that the clique set in $ \Gamma^{'}(R) $ contains $ 2 $ elements. Therefore $ \omega(\Gamma^{'}(R))=2$. If $ \Gamma^{'}(R)= mK_{1}$, which is a $ \bar{K}_{m} $ graph, then the clique set in $ \Gamma^{'}(R) $ contains $ m $ elements, since every isolated vertex is a 1-clique. It follows that $ \omega(\Gamma^{'}(R))=m$. 
\end{proof}
\begin{theorem}
	Let $ R $ be a finite commutative ring with unity consists of $ m $ distinct sets of mutual inverses $ \{r_{i},r_{j}\} $, where $ 1\leq i\leq j\leq m $ of order 1 or 2 respectively. Then $ \chi(\Gamma^{'}(R))\in \{1,2\}$.
\end{theorem}
\begin{proof}
Assume $ R $ is a finite commutative ring with with unity $ e $, that is $ R $ has $ n $ distinct elements as $ R=\{r_{i}:i=1,2,3,\cdots, n\}$. Let $ U(R) $ denote the set of units of $ R $, thus  $ U(R)=\{r_{i}: r_{i}\cdot r_{j} =e\} $. Since $ U(R) $ consists of $ m $ distinct sets of mutual inverses $ \{r_{i},r_{j}\} $, where $ 1\leq i\leq j\leq m $ of order 1 or 2, therefore by Theorem 3.6, the unity product graph is either $ \Gamma^{'}(R)= 2K_{1}+(m-2)K_{2}$ or  $ \Gamma^{'}(R)= 4K_{1}+(m-4)K_{2}$ or $ \Gamma^{'}(R)= mK_{1}$. If $ \Gamma^{'}(R)= mK_{1}$, which is a $ \bar{K}_{m} $, then the minimum number of colours needs to colour the vertices of the graph, is one. Hence, $ \chi(\Gamma^{'}(R))=1$. If the unity product graph is either in the form of $ \Gamma^{'}(R)= 2K_{1}+(m-2)K_{2}$ or  $ \Gamma^{'}(R)= 4K_{1}+(m-4)K_{2}$, then the minimum number of colours needs to colour the vertices of the graph is 2, since every $ K_{2} $ graph can be coloured by two different colours and at the same time the isolated vertices can be coloured by one of these two colours, thus $ \chi(\Gamma^{'}(R))=2$. 
\end{proof}
\begin{theorem}
	Let $ R $ be a division ring with $ Char(R)=0 $ and $ \Gamma^{'}(R) $ be the unity product graph associated with $ R $. Then the following holds:\\
	(1) $ \chi(\Gamma^{'}(R))=2$. \\
	(2) $ \omega(\Gamma^{'}(R))=2$.\\
\end{theorem}
\begin{proof}
		(1) Assume $ R $ is a division ring with $ Char(R)=0 $, that is $ R $ has infinite distinct elements including $ e $ as $ R=\{r_{i}\}_{i=1}^{\infty}$. Since $ R $ is a division ring, therefore every non-zero $ r_{i}\in R $ has a unique multiplicative inverse $ r_{j} $ such that $ r_{i}\cdot r_{j} =e $. This implies that $ |U(R)|=\infty $. Suppose that $ \Gamma^{'}(R) $ is the unity product graph of $ R $, thus  $ V( \Gamma^{'}(R))=U(R)$ and  $ E( \Gamma^{'}(R))=\{\{r_{i},r_{j}\}:  r_{i}\cdot r_{j} =e \ \textrm{for all}\ r_{i}\neq r_{j}\}$. Since $ e\in V( \Gamma^{'}(R))$, where $  e\cdot e=e $, this implies that there exists $ -e\in V( \Gamma^{'}(R))$ such that $ (-e)\cdot ( -e)=e $. Similarly, for every non-zero $ r_{i} $ which is not equal to $ e $ and $ -e $, there exist a unique $ r_{j} $ such that $ r_{i}\cdot r_{j} =e $. This shows that $ \Gamma^{'}(R) $ is the union of infinite copies of complete graph of order $ 2 $, $ K_{2} $ with two isolated vertices. Since every $ K_{2} $ graph can be coloured by two different colours and at the same time the two isolated vertices can be coloured by one of these two colours, thus $ \chi(\Gamma^{'}(R))=2$.\\
	(2) By the definition of a clique, the largest complete subgraph in $ \Gamma^{'}(R) $ is $ K_{2} $, this completes the proof with $ \omega(\Gamma^{'}(R))=2$.
\end{proof}
\begin{proposition}
	Let $ R $ be a ring and $ \Gamma^{'c}(R) $ be the complement unity product graph associated with $ R $. If $ R $ is a reduced ring with $ Char(R)=0 $ and $ |U(R)|=2 $, then the following holds: \\
	(1) $ \chi(\Gamma^{'c}(R))=2$. \\
	(2) $ \omega(\Gamma^{'c}(R))=2$.	
\end{proposition}
\begin{proof}
	(1) Assume  $ R $ is a reduced ring with $ Char(R)=0 $ and $ |U(R)|=2 $. This means that $ \Gamma^{'c}(R) $ is a complete graph of order 2, $ K_{2}$. Since the $ K_{2} $ graph can be coloured by two different colours, it follows that $ \chi(\Gamma^{'}(R))=2$.\\
	(2) The largest complete sub-graph in $ \Gamma^{'c}(R) $ is $ \Gamma^{'c}(R) $ itself which is $ K_{2} $, Hence,  $ \omega(\Gamma^{'}(R))=2$.\\
\end{proof}
\begin{proposition}
	Let $ R $ be a ring and $ \Gamma^{'c}(R) $ be the complement unity product graph associated with $ R $. If $ R $ is isomorphic to  $ R\cong \mathbb{Z}_{n}$, where $ n>2$ is a divisor of 24. Then the following holds: \\
	(1) $ \chi(\Gamma^{'c}(R))\in \{2,4,8\}$. \\
	(2) $ \omega(\Gamma^{'c}(R))\in \{2,4,8\}$.\\ 	
\end{proposition}
\begin{proof}
	Assume $ R $ is isomorphic to  $ R\cong \mathbb{Z}_{n}$, where $ n>2$ is a divisor of 24. By Proposition 3.3 (2), $ \Gamma^{'c}(R) $ is a complete graph. If $ n=3,4,6 $, then $ |V(\Gamma^{'c}(R))|=2 $, while for $ n=8,12 $ and $ n=24 $, then $ |V(\Gamma^{'c}(R))|=4 $ and $ |V(\Gamma^{'c}(R))|=8 $. This implies that $ \Gamma^{'c}(R) $ is a complete graph $ K_{m} $, where  $m= 2, 4, 8$. \\
	(1) Since the minimum number of colours need to colour the vertices of a complete graph, $ K_{m} $ is equal to $ m $, therefore  $ \chi(\Gamma^{'c}(R))=m$, where $m\in \{2,4,8\}$. \\
	(2) By the definition of a clique, the largest complete sub-graph in $ \Gamma^{'c}(R) $ is the graph itself which is $ K_{m} $. Hence,  $ \omega(\Gamma^{'c}(R))=m$, where $m\in \{2,4,8\}$.\\
\end{proof}
\begin{theorem}
	Let $ R $ be a division ring with $ Char(R)=p $,  where $ p\geq 5 $ is a prime. Further, let $ \Gamma^{'c}(R) $ be the complement unity product graph associated with $ R $, then the following holds:\\
	(1) $ \chi(\Gamma^{'c}(R))=|V(\Gamma^{'c}(R))|-n$, where $ n $ is the number of partite sets of order 2 in $ \Gamma^{'c}(R) $. \\
	(2) $ \omega(\Gamma^{'c}(R))=\frac{p+1}{2}$.\\
\end{theorem}
\begin{proof}
(1) Assume $ R $ is a division ring with $ Char(R)=p $, where $ p\geq 5 $ is a prime, then $ \Gamma^{'}(R) $ is a union of $ \frac{p-3}{2} $ copies of complete graph $ K_{2} $ with two isolated vertices.  Suppose that $ e $ and $ r_{p} $ are two isolated vertices in $ \Gamma^{'}(R) $ and every distinct non isolated pair of $ (r_{i},r_{j}) $ such that $ r_{i}\cdot r_{j}=e $ form a $ K_{2} $, thus the isolated vertices are adjacent to every elements in $ \Gamma^{'c}(R) $ and the endpoints vertices $ r_{i}$ and $ r_{j}$ of $ K_{2} $ are adjacent to all vertices in $ \Gamma^{'c}(R) $ but not to themselves. This implies that there are  $ p-2 $ edges incident to every isolated vertices $ e $ and $ r_{p} $, and $ p-3 $ edges incident to every non isolated vertices $ r_{i}$ in $ \Gamma^{'c}(R) $. By Theorem 3.7, $ \Gamma^{'}(R) $ is a complete $ \frac{p+1}{2} $ partite graph as $\Gamma^{'c}(R)= K_{2,2,2,\cdots,2,1,1} $. This shows that the minimum number of colours need to colour the vertices of $ \Gamma^{'c}(R) $  is equal to $ |V(\Gamma^{'c}(R))|-n $. Therefore, $ \chi(\Gamma^{'c}(R))=|V(\Gamma^{'c}(R))|-n$. \\  
(2) Since there are $ p-2 $ edges incident to every isolated vertices $ e $ and $ r_{p} $, and $ p-3 $ edges incident to every non isolated vertices $ r_{i}$ in $ \Gamma^{'c}(R) $. This implies that the largest complete sub-graph in $ \Gamma^{'c}(R) $ is a $ K_{m} $, where $ m=\frac{p-1}{2}+1$. Therefore, $ \omega(\Gamma^{'c}(R))=\frac{p+1}{2}$.\\
\end{proof}
\begin{theorem}
	Let $ R $ be a division ring with $ Char(R)=0 $ and $ \Gamma^{'c}(R) $ be the complement unity product graph associated with $ R $. Then the following holds:\\
	(1) $ \chi(\Gamma^{'c}(R))=\infty$. \\
	(2) $ \omega(\Gamma^{'c}(R))=\infty$.\\
\end{theorem}
\begin{proof}
	(1) Assume $ R $ is a division ring with $ Char(R)=0 $, that is $ R $ has infinite distinct elements including $ e $ as $ R=\{r_{i}\}_{i=1}^{\infty}$. By Theorem 5.6, $ \Gamma^{'}(R) $ is a graph that contains infinite copies of complete graph of order $ 2 $, $ K_{2} $ with two isolated vertices. Suppose that $ e $ and $ -e $ are two isolated vertices in $ \Gamma^{'}(R) $ and every distinct non isolated pair of $ (r_{i},r_{j}) $ such that $ r_{i}\cdot r_{j}=e $ form a $ K_{2} $, thus the isolated vertices are adjacent to every elements in $ \Gamma^{'c}(R) $ and the endpoints vertices $ r_{i}$ and $ r_{j}$ of $ K_{2} $ are adjacent to all vertices in $ \Gamma^{'c}(R) $ but not to themselves. This implies that on contrary there are infinite edges incident to every isolated vertices $ e $ and $ -e $, and infinite edges incident to every non isolated vertices $ r_{i}$ in $ \Gamma^{'c}(R) $. Hence the minimum number of colours need to colour the vertices of $ \Gamma^{'c}(R) $  is equal to $ \infty $. Therefore, $ \chi(\Gamma^{'c}(R))=\infty$. \\
	(2) By the definition of clique, the largest complete sub-graph in  $ \Gamma^{'c}(R) $ is a complete graph of order infinity, $ K_{\infty} $, therefore $ \omega(\Gamma^{'c}(R))=\infty$.
	\end{proof}

\section{Planarity and Hamiltonian of Unity Product Graph and its Complement}
\begin{theorem}
	Let $ R $ be a finite commutative ring with unity and  $ \Gamma^{'}(R) $ be the unity product graph associated with $ R $. Then $ \Gamma^{'}(R) $ is planar. 
\end{theorem}
\begin{proof}
Assume $ R $ is a finite commutative ring with unity, then $ R $  can be represented as $ R=\{r_{i}:i=1,2,3,\cdots, n\} $. By Theorem 3.6, the unity product graph is either $ \Gamma^{'}(R)= 2K_{1}+(m-2)K_{2}$ or  $ \Gamma^{'}(R)= 4K_{1}+(m-4)K_{2}$ or $ \Gamma^{'}(R)= mK_{1}$. This shows that no two of $ \Gamma^{'}(R) $ edges cross each other. Hence, $ \Gamma^{'}(R) $ is planar.   
 \end{proof}
\begin{theorem}
	Let $ R $ be a finite commutative ring with unity and  $ \Gamma^{'c}(R) $ be the complement unity product graph associated with $ R $. Then $ \Gamma^{'}(R) $ is planar if and only if $ |U(R)|\leq 4 $.
\end{theorem}
\begin{proof}
	Assume $ R $ is a finite commutative ring with unity. If $ |U(R)|\leq 4 $, then $ \Gamma^{'c}(R) $ is either a $ K_{m} $ or a $ K_{2,1,1} $ graph, where $ m\leq 4 $. This shows that $ \Gamma^{'c}(R) $ does not contain a subdivision of $ K_{5} $ or $ K_{3,3} $ as a subgraph. Hence, $ \Gamma^{'c}(R) $ is planar. Conversely, if $ \Gamma^{'c}(R) $ is a planar graph, we show that $ |U(R)|\leq 4 $. Suppose that $ |U(R)|\geq 4 $, then by Theorem 3.7, $ \Gamma^{'c}(R) $ is either a complete $ m $-partite graph $ K_{2,2,2,\cdots,2,1,1} $ or $ K_{2,2,2,\cdots,2,1,1,1,1} $ or a complete graph $ K_{8} $ that contain a subdivision of $ K_{5} $ as a subgraph, which is a contradiction. Hence, $ |U(R)|\leq 4 $.  
\end{proof}

\begin{theorem}
	Let $ R $ be a finite commutative ring with unity and  $ \Gamma^{'}(R) $ be the unity product graph associated with $ R $. Then $ \Gamma^{'}(R) $ is not Hamiltonian.
\end{theorem}
\begin{proof}
	 Assume $ R $ is a finite commutative ring with unity, then $ R $  can be represented as $ R=\{r_{i}:i=1,2,3,\cdots, n\} $. By Theorem 3.6, the unity product graph is either $ \Gamma^{'}(R)= 2K_{1}+(m-2)K_{2}$ or  $ \Gamma^{'}(R)= 4K_{1}+(m-4)K_{2}$ or $ \Gamma^{'}(R)= mK_{1}$. This shows that  $ \Gamma^{'}(R) $ contains no cycle cycle. Hence, $ \Gamma^{'}(R) $ is not Hamiltonian. 
\end{proof}

\begin{theorem}
	Let $ R $ be a finite commutative ring with unity and $ \Gamma^{'}(R) $ be the unity product graph associated with $ R $. Then $ \Gamma^{'c}(R) $ is Hamiltonian if and only if $ |U(R)|>2 $. 
\end{theorem}
\begin{proof}
	Assume $ R $ is a finite commutative ring with unity. If $ |U(R)|>2 $, then by Theorem 3.7, $ \Gamma^{'c}(R) $ is either a complete graph $ K_{4} $ or $ K_{8} $ or a complete $ m $-partite graph either   $ K_{2,2,2,\cdots,2,1,1} $ or $ K_{2,2,2,\cdots,2,1,1,1,1} $. This shows that there exists a cycle in $ \Gamma^{'c}(R) $ that visits every vertex of $ \Gamma^{'c}(R) $ exactly once. Hence, $ \Gamma^{'c}(R) $ is Hamiltonian. Conversely, if $ \Gamma^{'c}(R) $ is Hamiltonian, we show that $ |U(R)|>2 $. Let $ |U(R)|\leq 2 $, then $ \Gamma^{'c}(R) $ contains no cycle, which is a contradiction. Hence, $ |U(R)|>2 $.       
\end{proof}

\section{Conclusion}
 In this paper the unity product graph and its complement associated with some commutative rings with unity are investigated. The results obtained in this research show that the unity product graph associated with commutative rings with unity, which has at least two vertices, is always disconnected, while its complement graph is connected. However, if the commutative rings contain no unity element, these results are not always true. Some results are also established which determine the number of isolated vertices in unity product graph. Moreover, some properties of unity product graph and its complement in terms of girth, diameter, radius, dominating number, chromatic number and clique number, planarity and Hamiltonian are presented.

 \section*{Acknowledgement}
 The authors would like to acknowledge Universiti Teknologi Malaysia (UTM) and Research Management Centre (RMC) UTM for the financial funding through the UTM Fundamental Research Grant Vote No. 20H70 and  Fundamental Research Grant Scheme (FRGS1/2020/STG06/UTM/01/2). The first author would also  like to thank the Ministry of Higher Education (MOHE) of Afghanistan for his scholarship and Shaheed Prof. Rabbani Education University for the study leave. The third author would also like to thank UTM for his postdoctoral fellowship.
   


\end{document}